\documentclass[12pt]{amsart}
\usepackage{amsmath, amsthm, amssymb}
\usepackage{fullpage}
\usepackage{color}
\usepackage{hyperref}
\usepackage{soul}

\newtheorem{theorem}{Theorem}
\newtheorem{lemma}{Lemma}
\newtheorem{proposition}[lemma]{Proposition}

\newtheorem{remark}[lemma]{Remark}

\numberwithin{lemma}{section}

\numberwithin{equation}{section}

\newcommand{\R}{{\mathbb R}}

\renewcommand{\R}{\mathbb R}

\begin{document}

\title{Global bounds for the cubic nonlinear Schr\"odinger equation (NLS) in one space dimension}

\author{Mihaela Ifrim}
\address{Department of Mathematics, University of California at Berkeley}
\thanks{The first author was supported by the Simons Foundation}
\email{ifrim@math.berkeley.edu}

\author{ Daniel Tataru}
\address{Department of Mathematics, University of California at Berkeley}
 \thanks{The second author was partially supported by the NSF grant DMS-1266182
as well as by the Simons Foundation}
\email{tataru@math.berkeley.edu}

\begin{abstract}

  This article is concerned with the small data problem for the cubic
  nonlinear Schr\"odinger equation (NLS) in one space dimension, and
  short range modifications of it. We provide a new, simpler approach in order
  to prove that global solutions exist for data which is small in
  $H^{0,1}$.  In the same setting we also discuss the related problems
of obtaining a modified scattering expansion for the solution, as well as 
asymptotic completeness.

\end{abstract}

\maketitle

\section{Introduction}
We consider the cubic nonlinear Schr\"odinger equation (NLS)  problem 
in one space dimension
\begin{equation}\label{nls}
i u_t + \frac{1}{2}u_{xx} = \lambda u |u|^2, \qquad u(0) = u_0, 
\end{equation}
where $u$ is a complex valued function, $u:\mathbb{R}\times
\mathbb{R}\rightarrow \mathbb{C}$, and $\lambda= 1$ or $-1$
corresponding to the defocusing, respectively the focusing case.

Our results and proofs apply  equally to short range modifications of it
\begin{equation}\label{nls-pert}
i u_t +\frac{1}{2} u_{xx} = \lambda u |u|^2 + uF(|u|^2), \qquad u(0) = u_0 ,
\end{equation}
where $F$ satisfies
\begin{equation*}
|F(r)| \lesssim |r|^{1+\delta}, \qquad |F'(r)| \lesssim |r|^\delta, \qquad \delta > 0.
\end{equation*}
A common feature of these two equations is that they exhibit  Galilean invariance
as well as the phase rotation symmetry, both of which are used in our arguments.

The question at hand is that of establishing global existence and asymptotics 
for solutions to \eqref{nls} and then \eqref{nls-pert}, provided that the initial data 
is small and spatially localized. Traditionally this is done 
in Sobolev spaces of the form $H^{m,k}$, whose norms are defined by
\begin{equation*}
\Vert u\Vert _{H^{m,k}}^2:=
   \Vert  (1-\partial ^2_{x})^{\frac{m}{2}}u\Vert _{L^2}^2+\Vert  (1+\vert x\vert ^2)^{\frac{k}{2}}   u\Vert _{L^2}^2  ,  \qquad m,k \geq 0.
\end{equation*}

The problem \eqref{nls} is completely integrable, which allows one to
use very precise techniques, i.e., the inverse scattering method, to
obtain accurate long range asymptotics, even for large data in the
defocusing case.  These have the form
\[
u(t,x) \approx  t^{-\frac12} e^{\frac{ix^2}{2t} + \lambda |W(x/t)|^2 \log t}  W(x/t).
\]  
One notes that this is not linear scattering, but rather a modified
linear scattering.  Indeed, in work of Deift and Zhou \cite{DZ}, the
inverse scattering method is used to show that the above asymptotics
hold for data in $H^{1,1}$, not only for
\eqref{nls}, but also for \eqref{nls-pert} with a restricted range of
powers.

In the meantime, two alternate approaches have emerged, which do not depend on the 
complete integrability of the problem. The first, initiated by Hayashi and Naumkin \cite{HN},
and refined by Kato-Pusateri~\cite{KP}, is based on deriving an asymptotic equation 
for the Fourier transform of the solutions,
\[
\frac{d}{dt} \hat u(t,\xi) = \lambda  it^{-1}\, \hat u(t,\xi) | \hat{u}(t,\xi)|^2 + O_{L^{\infty}}(t^{-1-\epsilon}).
\]
This gives a result for data in $H^{1,1}$. 

A second approach, introduced by Lindblad-Soffer~\cite{LS}, is based on
deriving an asymptotic equation in the physical space along rays
\[
(t \partial_t + x \partial_x)   u(t,x) = \lambda i t^{-1}  u(t,x) | u(t,x)|^2 + O_{L^{\infty}}(t^{-1-\epsilon}).
\]
This argument yields a similar result, though the precise regularity at which this works is not 
computed, and is likely higher.

The method in the present paper in some sense interpolates between the two ideas above.
Instead of localizing sharply on either the Fourier or the physical side, we use a mixed 
wave packet style phase space localization, loosely inspired from the analysis in \cite{T-global}. 
The idea is that using wave packets one can produce a more accurate approximate
solution to the linear Schr\"odinger equation, and use that to test for the long range behavior
in the nonlinear equation.

Our interest in this problem arose from working on two dimensional
water waves, where a similar situation occurs.  There, a global result
was independently proved by Ionescu-Pusateri \cite{ip} and Alazard-Delort \cite{ad} using
methods based on the two ideas above. However, implementing either of
these strategies brings on considerable difficulties. Many of these
difficulties are bypassed by the authors in \cite{IT}, where a simpler
proof of the global result is given.

The present paper contains the implementation of the ideas in \cite{IT}
for the simpler problems \eqref{nls}, \eqref{nls-pert}. Our goal is
two-fold, namely to provide a simpler proof of the global result with
fewer assumptions, and also to give a more transparent introduction to
the work in \cite{IT}. Our main result is Galilean invariant:
  \begin{theorem}
    a) (Global well-posedness and decay) Consider either the equation
    \eqref{nls} or \eqref{nls-pert}, with initial data $u_0$ which is
    small in $H^{0,1}$,
\begin{equation} \label{main}
\|u_0\|_{ H^{0,1}} \leq \epsilon \ll 1.
\end{equation}
Then there exists a unique global solution
$u$ with regularity $e^{-\frac{it}{2} \partial_x^2} u\in C(\mathbb{R};H^{0,1}(\mathbb{R}))$
which satisfies the pointwise estimate
\begin{equation}\label{point}
\| u\|_{L^\infty} \lesssim \epsilon |t|^{-\frac12},
\end{equation}
as well as the energy bound
\begin{equation}\label{energy}
\| e^{\frac{-it}{2} \partial_x^2} u\|_{H^{0,1}} \lesssim \epsilon (1+t)^{C\epsilon ^2}.
\end{equation}

b) (Asymptotic behavior) Let $u$ be a solution to either \eqref{nls} or \eqref{nls-pert}
as in part (a). Then there exists a  function $W \in H^{1-C\epsilon^2}(\mathbb{R})$ such that 
\begin{equation}  
\label{asy}
\begin{split}
& u(x, t) = \frac{1} {\sqrt{t}} e^{\frac{ix^2}{2t}} 
 W(x/t) e^{i\log t\vert W(x/t) \vert ^2}+ err_{x}, 
 \end{split}
\end{equation}
\begin{equation}  
\label{asy1}
\begin{split}
& \hat{u}(\xi, t) =  e^{-\frac{it\xi^2}{2}} W(\xi) e^{i\log t\vert W(\xi) \vert ^2} + err_\xi ,
\end{split}
\end{equation}
where 
\begin{equation*}  
\begin{split}
& err_x \in \epsilon \left( O_{L^{\infty}}((1+t)^{-\frac{3}{4}+C\epsilon^2})
\cap O_{L^{2}_{x}}((1+t)^{-1+C\epsilon^2})\right), 
\end{split}
\end{equation*}
\begin{equation*}  
\begin{split}
& err_\xi \in \epsilon\left( O_{L^{\infty}}((1+t)^{-\frac14+C\epsilon^2})\cap O_{L^{2}_{\xi}}((1+t)^{-\frac12+C\epsilon^2})\right).
\end{split}
\end{equation*}

c) (Asymptotic completeness for small data)  Let $C$ be a large
universal constant. For each $W$ satisfying
\begin{equation*}
 \|W\|_{ H^{1+C\epsilon^2}(\mathbb{R})} \ll \epsilon \ll 1
\end{equation*}
there exists $u_0 \in H^{0,1}$ satisfying \eqref{main} so that \eqref{asy} and \eqref{asy1} hold
for the corresponding solution $u$ to  \eqref{nls} or \eqref{nls-pert}.
\end{theorem}

The next section contains the proof of the theorem. We begin with the
proof of part (a), which is a self contained argument.  The argument
for part (b) is based on a more careful analysis of the outcome of
(1a).  Finally, the proof of the asymptotic completeness is again a
self contained argument, which is a simpler lower regularity version
of the original result in \cite{KT}. Several remarks may be of interest:

\begin{remark}
Since one goal of this article is to present a clear and simple 
statement, the result and the proofs are done in the setting of $H^{0,1}$ data.
However, with some extra work, the same method will also work for data in $H^{0,s}$ 
with $\frac12 < s \leq 1$. 
\end{remark}

\begin{remark}
  One may ask whether one does not have $W \in H^1$, with a smooth one
  to one correspondence between $u_0 \in H^{0,1}$ and $W$. The work
  \cite{DZ} of Deift and Zhou shows that this is not the case, and
  that there is necessarily some logarithmic type correction to such a property.
We leave open the question of providing a direct proof of such a correspondence 
in a suitable functional setting.
\end{remark}


\section{Proof of the Theorem~\ref{main}}

\subsection{Local well-posedness}

While the equation \eqref{nls} is locally well-posed for data in $L^2$, working with data in 
$H^{0,1}$ requires a brief discussion.  The initial data space has norm 
\[
\|u_0\|^2_{H^{0,1}} = \|u_0\|_{L^2}^2 + \| x u_0\|_{L^2}^2.
\]
However, we cannot use this same space at later times since the weight
$x$ does not commute with the linear Schr\"odinger flow. Instead, we
introduce the vector field $L=x+it\partial_{x}$, which is the conjugate of $x$ 
with respect to the linear flow, $e^{\frac{it}{2} \partial_x^2} x = L e^{\frac{it}{2} \partial_x^2}$, as well as the
generator for the Galilean group of symmetries. Naturally we have  
\begin{equation*}
\begin{aligned}
\left[ i\partial_{t}+\frac{1}{2}\partial_{x}^2 \ , L \right] =0,\qquad L(\lambda u\vert u\vert ^2)=2\lambda \vert u\vert^2 Lu-\lambda u^2\overline{Lu}.
 \end{aligned}
\end{equation*}

Next, we state and prove a preliminary global result:

\begin{proposition}\label{local}
The equation \eqref{nls} is (globally) well-posed for initial data in $H^{0,1}$, in the sense that 
it admits a unique solution $u \in C(\R,L^2)$ such that $Lu \in C(\R,L^2)$. Further, such a solution 
is continuous away from $t = 0$, and satisfies $u \in C(\R\setminus\{0\}, L^\infty)$. 
Furthermore, near $t=0$ we have
\begin{equation}\label{point-small}
|u(t,x)| \lesssim t^{-\frac12} \|u_0\|_{H^{0,1}}.
\end{equation}
\end{proposition}

\begin{proof} We start with the $L^2$ well-posedness, which is based on the Strichartz
estimate for the linear inhomogeneous problem 
\[
(i\partial_{t}+\frac{1}{2}\partial_{x}^2)u = f, \qquad u(0) = u_0,
\]
which has the form
\begin{equation}\label{strichartz0}
\|u\|_{L^{\infty}_t L^2_x} + \|u\|_{L^4_t L^{\infty}_x} \lesssim \|u_0\|_{L^2} + \|f\|_{L^1_t L^2_x}.
\end{equation}
This allows us to treat the nonlinearity perturbatively and obtain the
unique local solution via the contraction principle in the space
$L_t^{\infty}(0,T; L^2_x) \cap L^4_t(0,T; L_x^{\infty})$ provided that $T$ is
small enough\footnote{This is exactly the scaling relation.}, $T \ll
\| u_0\|_{L^2}^4$. The local well-posedness in $L^2$ implies global
well-posedness due to the conservation of the mass $\|u\|_{L^2}^2$.

To switch to the $H^{0,1}$ data we need to write the equation for $Lu$,
which has the form
\begin{equation}\label{linearize}
(i\partial_{t}+\frac{1}{2}\partial_{x}^2)Lu = 2\lambda \vert u\vert^2 Lu-\lambda u^2\overline{Lu}.
\end{equation}
We remark that this is exactly the linearization of the  equation \eqref{nls}. 
The $L^2$ well-posedness of this problem also follows from the Strichartz estimate
\eqref{strichartz0}.

Finally, we consider pointwise bounds. Denoting $w =
ue^{-\frac{ix^2}{2t}}$, we have $ie^{-\frac{ix^2}{2t}} Lu =
it \partial_x w$. Hence, away from $t = 0$ we have $w \in C(\R
\setminus \{ 0\}; H^1)$, and the continuity property of $w$, namely $w
\in C_{loc}(\R \setminus \{ 0\}; C_0(\R))$, follows from the Sobolev
embedding $H^1(\R) \subset C_0(\R)$. Since $w$ has limit zero at infinity,
the similar property for $u$ also follows. Finally, the pointwise bound
\eqref{point-small} is a consequence of the Gagliardo-Nirenberg type estimate
\[
\|w\|_{L^\infty} \lesssim \|w\|_{L^2}^\frac12 \|\partial_x w\|_{L^2}^\frac12.
\]

\end{proof}

\subsection{ Wave packets and the asymptotic equation}
To study the global decay properties of solutions to \eqref{nls} and
\eqref{nls-pert}, we introduce a new idea, which is to test the
solution $u$ with wave packets which travel along the Hamilton flow. A
wave packet, in the context here, is an approximate solution to the
linear system, with $O(1/t)$ errors. Precisely, for each trajectory
$\Gamma_v := \{x = v t\}$, traveling with velocity $v$, we establish
decay along this ray by testing with a wave packet moving along the
ray.

To motivate the definition of this packet we recall some
useful facts. First, this ray is associated with waves which have 
spatial frequency 
\[
\xi_v := v = \frac{x}t.
\]
This is associated with the phase function
\[
\phi(t,x) := \frac{x^2}{2t}.
\]
Then it is natural to use as test functions  wave  packets  of the form 
\[
\Psi_{v}(t,x) :=  \chi\left(\frac{x - vt}{\sqrt{t}}\right)
e^{i\phi (t,x)} .
\]
Here we  take $\chi$ to be a Schwartz function. In other related
problems it might be more convenient to take $\chi$ with compact
support.  For normalization purposes we assume that
\[
\int \chi(y) dy = 1.
\]
The $t^\frac12$ localization scale is exactly the scale of wave packets
which are required to stay coherent on the time scale $t$. 
To see that these are reasonable approximate solutions we 
observe that we can compute
\begin{equation}\label{Ppsi}
(i \partial_t + \frac12 \partial_x^2) \Psi_v =  \frac1{2t} e^{i\phi} \partial_x \left[ 
t^\frac12 \chi'\left(\frac{x - vt}{\sqrt{t}}\right)  + 
i   (x-vt) \chi\left(\frac{x - vt}{\sqrt{t}}\right)\right],
\end{equation}
and observe that the right hand side has the same localization as
$\Psi_{v}$ and size smaller by a factor $t^{-1}$.  Thus one can think
of $\Psi_v$ as good approximate solutions for the linear Schr\"odinger
equation only on dyadic time scales $\Delta t \ll t$.

If one compares $\Psi_v$ with the fundamental solution to the linear 
Schr\"odinger equation, conspicuously the $t^{-\frac12}$ factor is missing.
Adding this factor does not improve the error in the interpretation of $\Psi_v$ as a good 
approximate solution, so we have preferred instead a normalization
which provides simpler ode dynamics for the function $\gamma$  defined below.

As a measure of the decay of $u$ along $\Gamma_v$ we use the function
\[ 
\gamma(t,v) := \int u \bar \Psi_{v} \, dx.
\]
For the purpose of proving part (a) of the theorem we only need to consider 
$\gamma$ along a single ray. However, in order to obtain the more precise 
asymptoptics in part (b) we will think of $\gamma$ as a function $\gamma(t,v)$.

We can also  express $\gamma(t,v)$ in terms of the 
Fourier transform of $u$,
\[
\gamma(t,v) = \int \hat u(t,\xi) \bar {\hat \Psi}(t,\xi) \, d\xi .
\]
Here a direct computation yields
\[
\begin{split}
\hat \Psi(t,\xi) =&  \frac{1}{\sqrt{2\pi }}\int e^{- i x \xi} e^{i \frac{x^2}{2t}}  \chi(t^{-\frac12} (x - vt)) \,dx\\
=&\frac{1}{\sqrt{2\pi }} e^{- i \frac{t\xi^2}2} e^{i \frac{t(\xi-v)^2}2}   \int e^{- i(x-vt)(\xi-v)} 
e^{i  \frac{(x-tv)^2}{2t}}  \chi(t^{-\frac12} (x - vt)) \, dx \\
= & t^{\frac12} e^{- i \frac{t\xi^2}2}  \chi_1(t^{\frac12}(\xi-v)),
\end{split}
\]
where $\chi_1 =  e^{i \frac{\xi^2}2} \widehat{e^{\frac{ix^2}2} \chi}$  is a Schwartz function
with the additional property that
\[
\int \chi_1(\xi)\,  d\xi = \int \chi(x) \, dx = 1.
\]
Then we can write
\begin{equation}\label{conv-xi}
\gamma(t,\xi) = e^{\frac{it \xi^2}2} \hat{u}(t,\xi) \ast_\xi  t^{\frac12} \chi_1(t^\frac12 \xi).
\end{equation}

Both the solution $u$ of \eqref{nls}  along the ray $\Gamma_v$  and its Fourier
transform evaluated at $v$ are  compared to $\gamma(t,v)$  as follows:
\begin{lemma}
\label{lema1}
The function $\gamma$ satisfies the bounds
\begin{equation}\label{bd-gamma}
\begin{split}
\|\gamma\|_{L^\infty} \lesssim t^\frac12 \|u\|_{L^\infty}, \qquad
\| \gamma \|_{L^2_v} \lesssim \|u\|_{L^2_x}, \qquad 
\| \partial_v  \gamma \|_{L^2_v} \lesssim \| Lu\|_{L^2_x}.
\end{split}
\end{equation}
We have the physical space bounds
\begin{equation}\label{diff-x}
\begin{split}
&\| u(t,v t) - t^{-\frac12}e^{i\phi(t,vt)} \gamma(t,v)\|_{L^2_v} \lesssim t^{-1} \|Lu \|_{L^2_x}, 
\\ 
& \| u(t,v t) - t^{-\frac12}e^{i\phi(t,vt)} \gamma(t,v)\|_{L^\infty} \lesssim t^{-\frac34} \|Lu \|_{L^2_x},
\end{split}\end{equation}
and the Fourier space bounds
\begin{equation}\label{diff-xi}
\begin{split}
& \| \hat u(t,\xi) -e^{-i\frac{t\xi^2}2} \gamma(t,\xi)  \|_{L^2_\xi} \lesssim t^{-\frac12}  \|Lu \|_{L^2_x}, 
\\
&\| \hat u(t,\xi) -e^{-i\frac{t\xi^2}2} \gamma(t,\xi)  \|_{L^\infty} \lesssim t^{-\frac14}  \|Lu \|_{L^2_x}.
\end{split}\end{equation}
\end{lemma}
\begin{proof}
  Denote $w:=e^{-i\phi} u$. Then $\partial_v w =i t \partial_x w =
 i t \partial_x ( e^{-i\phi} u) =i e^{-i\phi} Lu$, and  we can
  express $\gamma$ in terms of $w$ as a convolution with respect to
  the $v$ variable,
\begin{equation}\label{conv-x}
t^{-\frac12} \gamma (t,v) = w(t,vt) \ast_v t^\frac12 \chi(t^\frac12 v),
\end{equation}
where the kernel on the right has unit integral.
In other words,  $t^{-\frac12} \gamma (t,v)$ is a regularization of $w(t,vt)$ on the $t^{-\frac12}$
scale in $v$, or equivalently, a localization of $w(t,vt)$ to frequencies less than $t^{-\frac12}$.
Hence, via Young's inequality, we have the straightforward convolution bounds
\[
\| \gamma(t,v)\|_{L^\infty} \lesssim t^\frac12 \| w(t,vt) \|_{L^\infty} = t^\frac12 \| u\|_{L^\infty},
\]
\[
\| \gamma(t,v)\|_{L^2_v} \lesssim t^\frac12 \| w(t,vt) \|_{L^2_v} = \| u\|_{L^2_x},
\]
as well as 
\[
\| \partial_v \gamma(t,v)\|_{L^2_v} \lesssim t^\frac12 \| \partial_v w(t,vt) \|_{L^2_v} = \|L u\|_{L^2_x}.
\]
Here we have used the fact that  the $L_v^2$, $L^2_x$ norms are related by 
\[
\Vert f\Vert_{L^2_x}=t^{\frac{1}{2}}\Vert f\Vert_{L^2_v}.
\]

To bound the difference $t^{-\frac12} \gamma (t,v) -w(t,vt)$ we use the fact that 
the above kernel has unit integral to write
\begin{equation}\label{pdiff-x}
\begin{split}
|t^{-\frac12} \gamma (t,v) -w(t,vt)| = & \ 
 \left|\int (w(t,(v-z)t) -  w(t,vt)) \chi(t^\frac12 z) t^\frac12 \, dz\right|
\\ \leq  & \  \int |w(t,(v-z)t) -  w(t,vt)||\chi(t^\frac12 z)| t^\frac12  \, dz.
\end{split}
\end{equation}
To prove the pointwise bound in \eqref{diff-x} we use H\"older's inequality
to obtain
\begin{equation*}
|w(t,v t) - w(t,(v-z)t)| \lesssim |z|^\frac12 \|\partial_v w\|_{L^2_v},
\end{equation*}
which by \eqref{pdiff-x} leads to
\[
|e^{-i\phi} u(t,v t) - t^{-\frac12} \gamma(t,v)| \lesssim \|\partial_v w(t,vt)\|_{L^2_v}
\int  |z|^{\frac12}  t^\frac12 |\chi ( t^\frac12 z)| \, dz \approx
 t^{-\frac14} \| \partial_v w(t,vt)\|_{L^2_v} = t^{-\frac34} \|Lu \|_{L^2_x}.
\]
To prove the $L^2_v$ bound in \eqref{diff-x} we express the right hand side
in the last integrand in \eqref{pdiff-x} in terms of the derivative of $w$ to obtain
\[
 |t^{-\frac12} \gamma (t,v) -w(t,vt)| \lesssim
 \int_0^1 \int  |z| |\partial_v w(t,(v-hz)t)| t^\frac12 \chi(t^\frac12 z) t^\frac12 \, dz \, dh.
\]
Hence we can evaluate the $L^2$ norm as follows:
\[
\begin{split}
\|e^{-i\phi} u(t,v t) - t^{-\frac12} \gamma(t,v)\|_{L^2} \lesssim & \ \|\partial_v w(t,vt)\|_{L^2_v}
\int  |z|  t^\frac12 |\chi ( t^\frac12 z)| \, dz \\ \approx & \ 
 t^{-\frac12} \| \partial_v w(t,vt)\|_{L^2_v} = t^{-1} \|Lu \|_{L^2_x}.
\end{split}
\]
This concludes the proof of the bound \eqref{diff-x}. The estimate \eqref{diff-xi}
is obtained in a similar manner, but using \eqref{conv-xi} instead of \eqref{conv-x},
as well as the  relation  
\[
\| \partial_\xi (e^{\frac{it \xi^2}2} \hat u(t,\xi))\|_{L^2_\xi} = \| L u\|_{L^2_x} .
\]

\end{proof}
 By the previous Lemma~\ref{lema1} we can conclude that $\gamma$ is indeed a good approximation of $u$ along a ray,  but no information on the rate of decay of $\gamma$ was established.  Hence, the crucial next  step is to obtain an approximate ode dynamics for $\gamma(t,v)$:

\begin{lemma}
\label{lema2}
\label{ode}
If $u$ solves \eqref{nls} then we have
\begin{equation}\label{ode}
\dot \gamma(t,v) = -it^{-1} \lambda |\gamma(t,v)|^2 \gamma(t,v) - R(t,v),
\end{equation}
where the remainder $R$ satisfies
\begin{equation}
\label{LR}
\Vert R\Vert_{L_x^{\infty}}\lesssim t^{-\frac{1}{4}}\Vert Lu\Vert_{L^2_x} \left( t^{-1}+ \Vert u\Vert ^2_{L^{\infty}_x}\right), \qquad 
\Vert R\Vert_{L_v^2}\lesssim t^{-\frac12}\Vert Lu\Vert_{L^2_x} \left( t^{-1}+ \Vert u\Vert ^2_{L^{\infty}_x}\right).
\end{equation}
\end{lemma}

\begin{proof}
 A direct computation yields
\begin{equation*}
\begin{aligned}
\dot{\gamma }(t)= &  \int   u_t \bar{\Psi}_v+ u \bar \Psi_{vt} \, dx 
= \int i( \frac{1}{2}u_{xx}-\lambda u\vert u\vert^2)  \bar{\Psi}_v  + u \bar \Psi_{vt} \, dx 
\\ = & \int - i u  \overline{( i \partial_t + \frac12 \partial_x^2) \Psi_v}  
-  i\lambda u\vert u\vert^2\bar{\Psi}_v\, dx .
\end{aligned}
\end{equation*}
Using the relation \eqref{Ppsi} and integrating by parts we obtain
\begin{equation*}
\begin{aligned}
\dot{\gamma }(t)&=\int  i \frac1{2t}\partial_x ( 
t^\frac12 \chi'  + i  (x-vt) \chi)   e^{-i\phi}  u \, dx  - \int i\lambda u\vert u\vert^2\bar{\Psi}_v
 \, dx \\
& = - \int  \frac1{2t^2} ( 
 t^\frac12 \chi'  + i  (x-vt) \chi)   e^{-i\phi}  L u \, dx  - \int i\lambda u\vert u\vert^2\bar{\Psi}_v
 \, dx.
\end{aligned}
\end{equation*}
Hence we can  write an evolution equation for $\gamma(t)$ of the form
\begin{equation*}
\begin{aligned}
\dot{\gamma }(t,v)&=-i\lambda t^{-1} \vert \gamma(t,v)\vert^2\gamma(t)- R(t,v),
\end{aligned}
\end{equation*}
where  $R(t,v)$ contains error terms which are the contributions arising from using $\Psi_v$ as a good approximation of the solution of the linear Sch\"odinger equation, and also from substituting  $u$ by $\gamma$ in the cubic nonlinearity.
We write the remainder $R(t,v)$ as a sum of three quantities which can be easily bounded:
\begin{equation*}
\begin{aligned}
R(t,v):= & -\int \frac1{2t^2} ( t^\frac12 \chi'  + i  (x-vt) \chi)   e^{-i\phi}  L u \, dx
 -i\lambda \int  u \bar{\Psi}_v (\vert u\vert^2-\vert u(t,vt)\vert^2 )\, dx
\\ & \quad  \quad  \quad \quad   \quad  \quad  \,+ i \lambda \gamma(|u(t,vt)|^2 - t^{-1} |\gamma(t,v)|^2) \\
:= & R_1 + R_2 + R_3.
\end{aligned}
\end{equation*}

The integral $R_1$ is expressed as a convolution in $v$,
\[
R_1 = - \frac{1}{t}( t^\frac12 \chi'(t^\frac12 v) + i t v \chi(t^\frac12 v)) \ast_v (\partial_v w(t,vt)) .
\]
Hence, by H\"older's inequality we obtain the pointwise bound
\[
|R_1| \lesssim t^{-\frac34} \|  \partial_v w(t,vt)\|_{L^2_v} =  t^{-\frac54} \| Lu\|_{L^2_x} ,
\]
while estimating the convolution kernel in $L^1_v$ yields the $L_x^2$ bound
\[
\|R_1\|_{L^2_v} \lesssim t^{-1}  \|  \partial_v w(t,vt)\|_{L^2_v} =  t^{-\frac32} \| Lu\|_{L^2_x} .
\]
Since $|u| = |w|$, the second term $R_2:= -i\lambda \int  u \bar{\Psi}_v (\vert u\vert^2-\vert u(t,vt)\vert^2 )\, dx $ is bounded by
\[
\begin{split}
|R_2(t,v)| \lesssim & \ \| u\|_{L^\infty}^2 \int |\chi(t^{-\frac12}(x-vt))| (|w(t,x) - w(t,vt)|) \, dx
\\ = & \  t^\frac12 \| u\|_{L^\infty}^2 \int |\chi(t^{\frac12}z)| (|w(t,(v-z)t) - w(t,vt)|) t^\frac12 \, dz,
\end{split}
\] 
where the last integrand is the same as in \eqref{pdiff-x}.
Then $R_2$ is estimated exactly as in the proof of \eqref{diff-x} following \eqref{pdiff-x}.

Finally, for $R_3$ it suffices to combine the estimates \eqref{bd-gamma} and \eqref{diff-x}.

\end{proof}

\subsection{Proof of the global well-posedness result.}
From Proposition~\ref{local} we know that a global solution exists, so
it remains to establish the bounds \eqref{energy} and
\eqref{point}. Proposition~\ref{local} also shows that $\|
u(t)\|_{L^\infty}$ is continuous in time away from $t =0$.  Then a
continuity argument implies that it suffices to prove these bounds
under the additional bootstrap assumption:
\begin{equation}
\label{boot}
\Vert u\Vert_{L^{\infty}}\leq D\epsilon \vert t\vert ^{-\frac{1}{2}},
\end{equation}
where $D$ is a large constant such that $1\ll D \ll
\epsilon^{-1}$. Then we want to prove the energy bound \eqref{energy},
and then show that \eqref{point} holds with an implicit constant which
does not depend on $D$.

\textit{The energy estimate for $Lu$:}  
To advance frome time $0$ to time $1$ we use the local well-posedness result above.
This gives
\[
\| Lu(1)\|_{L^2} \lesssim \| xu(0)\|_{L^2} \leq \epsilon.
\]
To move forward in time past time $1$ we use energy estimates
in \eqref{linearize} and then  \eqref{boot} to obtain
\begin{equation*}
\begin{aligned}
\Vert Lu(t)\Vert_{L^2}&\leq\Vert Lu(1)\Vert_{L^2}+ \int_{1}^t\Vert u(s)\Vert ^2_{\infty}\Vert Lu(s)\Vert_{L^2}\, ds\\
&\leq\Vert Lu(1)\Vert_{L^2}+D^2\epsilon ^2\int_{1}^t s^{-1}\Vert Lu(s)\Vert_{L^2}\, ds.
\end{aligned}
\end{equation*}

Applying Gronwall's inequality  gives 
\begin{equation}
\label{L}
\Vert Lu(t)\Vert_{L^2}\lesssim \epsilon (1+t)^{D^2\epsilon ^2},
\end{equation}
which, combined with the conservation of mass, leads to
\begin{equation*}
\Vert e^{-\frac{it}{2}\partial_{x}^2}u(t)\Vert_{H^{0,1}}\lesssim \epsilon (1+t)^{D^2\epsilon ^2}.
\end{equation*}

\textit{The pointwise decay bound:} 
From the bound \eqref{diff-x} in Lemma~\ref{lema1} and \eqref{L} we get 
\begin{equation*}
\begin{aligned}
 \Vert e^{-i\phi}u-t^{-\frac{1}{2}}\gamma \Vert_{L_x^{\infty}}\lesssim
 t^{-\frac{3}{4}}\Vert Lu\Vert_{L^2_x}\lesssim\epsilon (1+t)^{-\frac{3}{4}+D^2\epsilon^2},
 \end{aligned}
 \end{equation*}
so it remains to estimate $\gamma$. At time $t = 1$ we can use \eqref{point-small} 
and the pointwise part of \eqref{bd-gamma} to conclude that 
\begin{equation*}
\| \gamma (1,v)\|_{L^\infty} \lesssim \epsilon.
\end{equation*}
On the other hand, using our bootstrap assumption \eqref{boot} and the $L^2$ bound \eqref{L}
in  Lemma~\ref{lema2} we obtain a good bound for $R(t,v)$, namely
 \begin{equation*}
\| R(t,v) \|_{L^\infty} \lesssim \epsilon ( 1 + D^2 \epsilon^2)  t^{-\frac{5}{4}+D^2\epsilon^2}.
 \end{equation*}
 Then integrating in \eqref{ode} we obtain
\begin{equation*}
\begin{aligned}
|\gamma(t,v)| \leq |\gamma(1,v)| + \int_{1}^t |R(s,v)| ds  \lesssim 
\epsilon (1+D^2\epsilon^2),
\end{aligned}
 \end{equation*}
 which leads to
 \begin{equation*}
\begin{aligned}
\vert u\vert \lesssim (\epsilon +D^2\epsilon^3)\vert t\vert^{-\frac{1}{2}}.
\end{aligned}
 \end{equation*}
 Under the constraint $1\ll D\ll \epsilon^{-1}$ we obtain
 \eqref{point}, and conclude the bootstrap argument.


\subsection{The asymptotic expansion of the solution}

To construct the asymptotic profile $W$ we use the ode in  Lemma~\ref{lema2}
for $\gamma(t,v)$. The inhomogeneous term $R(t,v)$ is estimated in $L^\infty$ 
and $L^2_v$ by combining  \eqref{LR} with \eqref{point} and \eqref{L} to obtain 
\begin{equation}\label{Rbd}
\| R(t,v)\|_{L^\infty} \lesssim \epsilon t^{-\frac{5}{4}+D^2\epsilon^2}, 
\qquad \| R(t,v)\|_{L^2_v} \lesssim \epsilon t^{-\frac{3}{2}+D^2\epsilon^2}.
 \end{equation}
The ODE for $\gamma$, namely
\begin{equation*}
\dot{\gamma}(t) = -\frac{i}{t}\vert \gamma (t)\vert ^2 \gamma (t)-R(t,v)
\end{equation*}
 can be explicitly solved in polar
coordinates. Since $R(t,v)$ in uniformly integrable in time, it
follows that for each $v$, $\gamma(t,v)$ is well approximated at infinity by a
solution to the unperturbed ODE corresponding to $R_1=0$, 
in the sense that
\begin{equation}
\label{daniel}
\gamma (t,v)=W(v) e^{i\vert W(v) \vert^2 \log t }+O_{L_v^{\infty}}(\epsilon t^{-\frac{1}{4}+D^2\epsilon^2}).
\end{equation}
Integrating the $L^2_v$ part of \eqref{Rbd}  leads to a similar $L^2_v$ bound
\begin{equation}
\label{daniel2}
\gamma (t,v)=W(v) e^{i\vert W(v) \vert^2 \log t }+O_{L_v^{2}}(\epsilon t^{-\frac{1}{2}+\epsilon ^2D^2}).
\end{equation}
Then the asymptotic expansions in \eqref{asy}, \eqref{asy1} follow directly from \eqref{diff-x}
and \eqref{diff-xi}, where $\|Lu\|_{L^2}$ is bounded as in \eqref{L}.

It remains to establish the regularity of $W$. By conservation of mass we have 
\[
\|u(0)\|_{L^2_x} = \|u(t)\|_{L^2_x} = t^{\frac12} \| w(t,vt)\|_{L^2_v}.
\]
 Hence by  \eqref{diff-x} and \eqref{daniel2}
we obtain 
\begin{equation*}
\Vert W\Vert _{L^2_x}=\Vert u\Vert _{L^2_x}.
\end{equation*}
On the other hand, from \eqref{daniel} and \eqref{daniel2} we get
\[
\| W(v) - \gamma(t,v) e^{-i |\gamma(t,v)|^2 \log t}\|_{L^2_v} \lesssim  \epsilon t^{-\frac12+D^2 \epsilon^2} \log t ,
\]
while by \eqref{bd-gamma} and \eqref{L} we have
\[
\| \partial_v [\gamma(t,v) e^{-i |\gamma(t,v)|^2 \log t}] \|_{L^2_v} \lesssim \epsilon t^{D^2 \epsilon^2}  \log t .
\]
It follows that  for all large $t$ we have
\begin{equation*}
W( v)=O_{H^1_v}(\epsilon t^{D^2 \epsilon^2}\log t)+O_{L^2_v}(\epsilon t^{-\frac{1}{2}+D^2 \epsilon^2}\log t),
\end{equation*}
 so by interpolation we obtain for large enough $C$ the regularity 
\[
\|W\|_{H^{1-C\epsilon^2}_v} \lesssim \epsilon.
\]


\subsection{ The asymptotic completeness problem}

Here we solve the problem from infinity. For convenience, throughout this section, 
we set $\lambda =1$.  The naive idea would be to start with the asymptotic profile 
\[
u_{asymptotic} = \frac{1}{\sqrt{t}}  e^{\frac{ix^2}{2t}}  {W}( x/t) e^{i |{W}(x/t)|^2\log t},
\]
and correct this to an exact solution $u$ to the cubic NLS \eqref{nls}, by perturbatively 
solving the equation for the difference from infinity. However, as defined above, the 
function $u_{asymptotic} $ does not have enough regularity in order for it to be a 
good approximate solution. To remedy this, we replace $W$ in the above formula 
with a regularization of $W$ on the time dependent scale, namely
\[
\mathcal{W}(t,v):= W_{<t^\frac12}( v),
\]
which selects the frequencies less than $t^\frac12$ in $W$.
This is the analogue of the function $\gamma$ defined for forward problem,
with the same time dependent regularization scale. Then our approximate solution 
is 
\[
u_{app} = \frac{1}{\sqrt{t}} e^{\frac{ix^2}{2t}} \mathcal{W}( t, x/t) e^{i |\mathcal{W}(t,x/t)|^2\log t}.
\]
To start with we make the more general assumption that 
\begin{equation}\label{W-bound}
\| W\|_{H^{1+2\delta}_v} \leq M, \qquad M,\delta > 0, \qquad \delta \gg M^2.
\end{equation}
 Then by Bernstein's inequality we have the bounds 
\[
\| \mathcal{W}(t,v) - W(v)\|_{L^2_v} \lesssim M t^{-\frac12- \delta},
 \qquad \| \mathcal{W}(t,v) - W(v)\|_{L^\infty} \lesssim M t^{-\frac14- \delta} ,
\]
which imply that the functions $u_{asymptotic}$ and $u_{app}$ 
are equally good as  asymptotic profiles,
\[
\| u_{asymptotic} - u_{app}\|_{L^2_x} \lesssim M t^{-\frac12-\delta},
 \qquad \| u_{asymptotic} - u_{app}\|_{L^\infty} \lesssim \epsilon t^{-\frac14-\delta}. 
\]

To find the exact solution $u$  matching $u_{app}$ at infinity we denote by $f$ 
the error
\begin{equation}
\label{f-def}
f = (i \partial_t +\frac{1}{2}\partial^2_{x}) u_{app} - u_{app} |u_{app}|^2,
\end{equation}
 and then solve for the diffrence $v = u - u_{app}$
\begin{equation*}
(i \partial_t +\frac{1}{2} \partial^2_{x} )v =  (u_{app} +v)\, |u_{app} + v|^2 - u_{app} |u_{app}|^2 - f.
\end{equation*}  
The $u_{app}$-cubic term cancels, and we are left with 
\begin{equation}
\label{eq3}
\begin{aligned}
&(i \partial_t + \frac{1}{2}\partial^2_x)v = N(v,u_{app}) - f, \qquad v(\infty) = 0,
\end{aligned}
\end{equation}
where
\[
N(v,u_{app}) =  v\vert v\vert ^2 + v^2 \bar{ u}_{app} +2\vert v\vert ^2u_{app}+2 v \vert u_{app}\vert ^2  +\bar{v}u^2_{app}.
\]
The solution operator for the inhomogeneous Schr\"odinger equation with zero Cauchy data 
at infinity 
\[
(i\partial_{t}+\frac{1}{2}\partial_{x}^2)v = f, \qquad u(\infty) = 0,
\]
is given by 
\begin{equation*}
v(t)=i\lambda \int _{t}^{\infty}e^{\dfrac{(t-s)\partial_x^2}{2}}  f(s)  \, ds:=\Phi f .
\end{equation*}

Hence the equation \eqref{eq3} is rewritten in the form
\begin{equation}\label{v-fp}
v = \Phi N(v,u_{app}) - \Phi f .
\end{equation}
We will solve this via the contraction principle,  using the energy/Strichartz type bound \eqref{strichartz0}
\begin{equation}
\|\Phi f\|_{L^{\infty}_t(T,\infty; L_x^2)} + \|\Phi f\|_{L^4_t(T,\infty; L^{\infty}_x)} \lesssim  \|f\|_{L^1_t(T,\infty; L_x^2)}.
\end{equation}
The equation for $v$ will be solved in a function space $X$ defined
using the above $L^{\infty}_t L^2_x$ and $L^4_t L^{\infty}_x$ norms, with
appropriate time decay. Precisely, we set
\[
\Vert v\Vert_{X}:=\sup_{T \geq 1} \frac{T^{\frac{1}{2}+\delta}}{(1+M^2
  \log t)^2}\left( \| v \|_{L_t^\infty(T,2T; L_x^2)} + \| v
  \|_{L_t^4(T,2T; L_x^{\infty})} \right).
\]
We also want a bound for $Lv$, for which we need to use the larger space
 $\tilde X$, whose  norm carries a different time decay weight,
\[
\Vert w\Vert_{\tilde X}:=\sup_{T \geq 1} \frac{T^{\delta}}{(1+M^2 \log
  t)^3}\left( \| w \|_{L_t^\infty(T,2T; L_x^2)} + \| w \|_{L_t^4(T,2T;
    L_x^{\infty})} \right).
\]

The first task at hand is to estimate the contribution of the inhomogeneous term $f$.
This is done in the following 

\begin{lemma}\label{l:f}
Assume that \eqref{W-bound} holds with $\delta \gtrsim M^2$. Then $f$ defined by 
\eqref{f-def} satisfies the following estimates:
\begin{equation}\label{f-bd}
\| \Phi f\|_{X} + \| \Phi Lf\|_{\tilde X} \lesssim M.
\end{equation}
\end{lemma}
We postpone the proof of the lemma in order to conclude first the 
proof of the main result.  We succesively consider the equation for $v$ 
and the equation for $Lv$.

\bigskip

{\em (i) The equation for $v$ in $L^2$.} In view of \eqref{f-bd}, in
order to solve the equation \eqref{v-fp} in $X$ using the contraction
principle we need to show that the map $v \to \Phi N(v,u_{app})$ maps
$X$ into $X$ with a small Lipschitz constant for $v$ in a ball of
radius $CM$, where $1 \ll C \ll M^{-1}$.  Then we obtain a solution $v$ 
satisfying 
\begin{equation}\label{vbound}
\|v\|_{X} \lesssim M.
\end{equation}
Using the linear bound
\eqref{strichartz0}, it suffices to show that
\begin{equation}
\| N(v_1,u_{app}) - N(v_2,u_{app})\|_{L^1_t(T,\infty;L_x^2)}
\lesssim \|v_1-v_2\|_{X} (M+ \|v_1\|_{X}^2 + \|v_2\|_X^2) .
\end{equation}
For simplicity we consider the case $v_2=0$ and show that
\begin{equation}\label{Nv}
\| N(v,u_{app}) \|_{L^1_t(T,\infty;L_x^2)}
\lesssim M \|v\|_{X} + \|v\|_{X}^3. 
\end{equation}
The general case is identical.
To bound $\|  N(v,u_{app}) \|_{L^1_t(T,\infty;L^2_x)}$ we 
 we divide $\left[ T, \infty \right) $ into dyadic
subintervals, estimate $N(u_{app},v;f)$ in each such interval, and then
sum up.  For the terms in $N$ we succesively compute 
\begin{equation}
\label{m1}
\begin{aligned}
\Vert v\vert u_{app}\vert^2\Vert_{L^1_t (T,2T;L^2_x)}&\lesssim T \Vert u_{app}\Vert ^2_{L^{\infty}([T,2T]\times \mathbb R)}\Vert v\Vert_{L_t^{\infty}(T,2T;L^2_x)} \lesssim M^2 T^{-\frac{1}{2}-\delta}(1+M^2 \log t)^2\Vert v\Vert_{X} ,
\end{aligned}
\end{equation}
\begin{equation}
\label{m2}
\begin{aligned}
\Vert \vert v\vert^2u_{app}\Vert_{L^1_t (T,2T;L^2_x)}\lesssim & \
 T^\frac34 \Vert u_{app}\Vert_{L^\infty([T,2T]\times \mathbb R)}
\Vert v\Vert_{L_t^{\infty}(T,2T;L^2_x)} \Vert v\Vert_{L_t^{4}(T,2T;L^{\infty}_x)} 
\\ \lesssim &  \ M  T^{-\frac{3}{4}+2\delta}(1+M^2 \log t)^4\Vert v\Vert^2_{X},
\end{aligned}
\end{equation}
respectively
\begin{equation}
\label{m3}
\begin{aligned}
  \Vert v\vert v\vert ^2\Vert_{L^1_t (T,2T;L^2_x)}&\lesssim
  T^{\frac{1}{2}}\Vert v\Vert_{ L^{\infty}_t (T,2T;L^2_x)} \Vert
  v\Vert^2_{L^4_t (T,2T;L^{\infty}_x)}  \lesssim
  T^{-1-3\delta}(1+M^2 \log t)^6 \| v\|_X^3 .
\end{aligned}
\end{equation}
Thus, \eqref{Nv} follows.

\bigskip

{\em (ii) The equation for $Lv$ in $L^2$.} 
 Applying  $L$  to \eqref{eq3} we obtain 
\begin{equation*}
\begin{aligned}
(i \partial_t + \frac{1}{2}\partial^2_x )Lv &= LN(v, u_{app}) - Lf.
\end{aligned}
\end{equation*}
Then for $Lv$ we seek to solve the linear problem
\[
Lv = \Phi(LN(v,u_{app})) - \Phi Lf
\]
in the space $\tilde X$. The bound for $\Phi Lf$ is provided by 
\eqref{f-bd}.  We expand $LN(v,u_{app})$ as
\begin{equation*}
\begin{aligned}
LN(v, u_{app}\, ;f ):&=L( v\vert v\vert ^2) + L(v^2 \bar{ u}_{app}) +2L(\vert v\vert ^2u_{app})+2 L(v \vert u_{app}\vert ^2)  +L(\bar{v}u^2_{app})- Lf\\
& = Q (Lv) + g - Lf,
\end{aligned}
\end{equation*}
where the linear part $Q(Lv)$, respectively the inhomogeneous term $g$ are given by
\begin{equation*}
\begin{aligned}
&Q(Lv)  :=2\vert v\vert^2Lv-v^2\overline{Lv}+2\bar{u}_{app}vLv-u^2_{app}\overline{Lv}+ 
\bar{v}u_{app}Lv -vu_{app}\overline{Lv}+\vert u_{app}\vert^2Lv ,\\
&g  := 2u_{app}\bar{v}Lu_{app}-v^2\overline{Lu}_{app} +\vert v\vert ^2Lu_{app} +v\bar{u}_{app}Lu_{app}-vu_{app}\overline{Lu}_{app}.
\end{aligned}
\end{equation*}
We can use again \eqref{strichartz0}, so it remains to estimate $Q(Lv)$
and $g$ in $L^1_t(T,\infty ; L^2_x)$.  For $u_{app}$ we make use only of
the pointwise bound  $\Vert u_{app}\Vert_{L^{\infty}} \lesssim M t^{-\frac12}$ and the
$L_x^2$ bound for $Lu_{app}$
\[
\Vert Lu_{app}\Vert_{L_x^2}\lesssim M(1+M^2\log t),
\]
while for $v$ we use the $X$ norm bound \eqref{vbound}.
The same type of analysis as in the proof of 
\eqref{m1}-\eqref{m3} leads to  the estimate
\[
\| Q(Lv)\|_{L^1_t(T,2T; L^2_x)} \lesssim M^2 T^{-\delta} (1+M^2 \log T)^3 \|Lv\|_{\tilde X},
\]
where the worst term in $Q(Lu)$ is the last one. 
After dyadic summation this yields 
\[
\| Q(Lv)\|_{L^1_t(T,\infty; L_x^2)} \lesssim \delta^{-1} M^2 T^{-\delta} (1+M^2 \log T)^3 \|Lv\|_{\tilde X}.
\]
This is where we need the condition $\delta \gg M^2$ both in order to
have a good dyadic summation, and in order to gain a small Lipschitz constant.
Next we bound $g$ in $L^1_t(T,\infty; L^2_x)$; this is better since we use 
at least one $v$ norm, and we  obtain
\[
\| g \|_{L^1_t(T,\infty; L^2_x)} \lesssim M^3 T^{-\frac14-\delta} (1+M^2 \log T)^3.
\]
The proof of the theorem is concluded, modulo the proof of Lemma~\ref{l:f}, which follows.

\begin{proof}[Proof of Lemma~\ref{l:f}]
We first compute $f$. For that we need the time derivative of $\mathcal{W}$,
\[
\partial_t \mathcal{W}(t,v) = t^{-1} W_{t^\frac12}( v),
\]
where $W_{t^\frac12}$ is obtained from $W$ via a zero order multiplier
which is localized exactly at dyadic frequency $t^{\frac12}$. Then we
can write
\begin{equation*}
\begin{aligned}
f  = \frac{1}{t^\frac12} e^{\frac{ix^2}{4t}} e^{i \log t |\mathcal{W} |^2}
 & \left\lbrace   \frac{1}{t} \left[  W_{t^\frac12} + 2i \mathcal{W} \log t \Re  (W_{t^\frac12}\bar {\mathcal{W}})\right]  \right.
\\ & + \frac{1}{t^2} \left[  \mathcal{W}''  +2 i \mathcal{W} \log t \Re  (\mathcal {W''} \bar {\mathcal{W}})- 4 \mathcal{W} \left( \log t \Re  (\mathcal{W}' \bar {\mathcal{W}})\right) ^2\right] 
\\ & \left. + \frac{1}{t^2} \left[ 2i \mathcal {W}' \log t \Re  (\mathcal {W}' \bar {\mathcal{W}})   + 2i \mathcal{W} \log t |\mathcal {W}'|^2\right] \right\rbrace ,
\end{aligned}
\end{equation*}
where $\mathcal{W}'$ and $\mathcal{W}''$ denote the first and the
second derivative with respect to $v$.  The expression for $Lf$ is
computed from this using the observation that
$L(e^{\frac{ix^2}{2t}}g(x/t))=ie^{\frac{ix^2}{2t}}\partial_{v}g(x/t)$.
From \eqref{W-bound} we have the $L^{\infty}$ and $L^2_v$ and bounds
\begin{equation}
\label{baiatul}
\begin{aligned}
& \Vert\mathcal{W} \Vert_{L^\infty}\lesssim M, \quad   \Vert \mathcal{W}' \Vert_{L^\infty}
\lesssim M t^{\frac14 - \delta},
\\
\quad  \Vert \mathcal{W}' \Vert_{L^2_v}\lesssim M,
 & \quad  \quad \Vert \mathcal{W}''\Vert_{L^2_v}\lesssim Mt^{\frac{1}{2}-\delta}, \quad  \quad \Vert \mathcal{W}'''\Vert_{L^2_v}\lesssim Mt^{1-\delta},
\\
 &\Vert W_{t^{\frac{1}{2}}}\Vert_{L^2_v}\lesssim Mt^{-\frac{1}{2}-\delta}, \quad \Vert W'_{t^{\frac{1}{2}}}\Vert_{L^2_v}\lesssim M t^{-\delta}.
  \end{aligned}
\end{equation}
Using these bounds it is easy to see that the following estimates hold 
\begin{equation}
\label{L1}
\|f\|_{L_x^2} \lesssim M t^{-\frac32-\delta} ( 1+ M^2 \log t)^2,
 \quad \|Lf\|_{L_x^2} \lesssim  M t^{-1-\delta}( 1+ M^2 \log t)^3.
\end{equation}
Then the  bound for $\Phi f$ in \eqref{f-bd} follows easily by time integration
and \eqref{strichartz0}.  Unfortunately, a  direct integration in the bound for $Lf$ in \eqref{L1}
yields an extra $\delta^{-1}$ factor,
\[
\| Lf \|_{L^1_t(T,\infty; L_x^2)} \lesssim \delta^{-1} M t^{-\delta} (1+M^2 \log t)^3 ,
\]
so the  bound for $\Phi Lf$ in \eqref{f-bd} cannot be obtained directly.

To improve on this, we first peel off the better part of $L f$, which
includes all terms which do not contain either of the factors
$W_{t^{\frac12}}$,  $\mathcal{W}''$.  Precisely, we set
\[
h := 
\frac{1}{t^\frac32} e^{\frac{ix^2}{2t}} \partial_vZ,
\]
where the expression of $Z$ is given by
\[
Z(t,v)=  e^{i \log t |\mathcal{W} |^2} \left( W_{t^{\frac{1}{2}}}+ 2i \mathcal{W} \log t \Re  (W_{t^\frac12}\bar {\mathcal{W}})+\frac{1}{t}\mathcal{W}'' +2i\frac{1}{t}\mathcal{W}\log t \Re (\mathcal{W}''\bar{\mathcal{W}})\right)   .
\]
The function $Z$ is essentially localized around frequency
$t^{\frac12}$;  this is seen in the estimates below for $Z$, which are computed in terms of the
$L^2_v$-norm of $W_{\leq t^{\frac12}}$:
\begin{equation}\label{Zest}
 \Vert \partial^j_v Z\Vert_{L^2_v}\lesssim t^{-1+\frac{j}{2}}(1+M^2\log t)^{j+1}\Vert W_{\leq t^{\frac12}}''\Vert_{L^2_v}, \quad j=0,1,2.
\end{equation}
Since the regularity of $W$ is $H^{1+\delta}$, this shows that the map
from $W$ to $Z$ is mostly diagonal with respect to frequencies, with
rapidly decaying off-diagonal tails.

The difference $Lf-h$ can be shown to have better time decay, 
\[
\| Lf - h\|_{L_x^2} \lesssim M^3 t^{-\frac54-\delta} (1+M^2 \log t)^3 ,
\]
which is stronger than needed. It remains to consider the output of
$h$, for which it is no longer enough to obtain a fixed time $L^2$
bound and then integrate it in time.  Instead, we consider $\Phi h$
directly.

To estimate $\Phi h$  we first compute the Fourier transform of $h$,
\[
\hat h(\xi) = \frac{1}{t^\frac32} \int  e^{-ix\xi}   e^{\frac{ix^2}{2t}} \partial_vZ(t,x/t)\, dx
=  \frac{1}{t^\frac12}  e^{\frac{it \xi^2}{2}} \int     e^{\frac{it (\xi-v)^2}{2}} \partial_v Z(t,v) \, dv.
\]
Interpreting the last integral as a convolution, we compute its
pullback to time zero, 
\[
(e^{\frac{ it \partial_x^2}{2}} h)(t,x) =t^{-1} e^{\frac{i x^2}{2t}} \widehat{(\partial_vZ)}(t,x)  = t^{-1} e^{\frac{i x^2}{2t}} ix\hat Z(t,x),
\]
which, in view of \eqref{Zest},  is mainly concentrated in the dyadic region $x \approx t^\frac12$.
Then the  solution to the backward Schr\"odinger equation is 
\[
\Phi h(t) = e^{-\frac{it\partial_x^2}{2}} z(t), \qquad 
z(t,x) = ix \int_t^{\infty}  s^{-1} e^{\frac{i x^2}{2s}} \hat Z(s,x) \,ds
\]
Now we take advantage of the fact that, in the above integral,
dyadic regions in $t$ essentially contribute to different dyadic regions in $x$. This shows that
\[
\| t^{-1} e^{\frac{i x^2}{2s}} x\hat Z(s,x) \|_{l^2 L^1_t(T,\infty; L^2_x)} \lesssim \| W'_{\geq T^\frac12} \|_{L^2_v} +T^{-\frac12 }\Vert W''_{\leq T^{\frac12}}\Vert_{L^2_v}
\lesssim M T^{-\delta} (1+M^2 \log T)^3,
\] 
where the $l^2$ norm is taken with respect to dyadic regions in frequency.
After time integration this implies that 
\[
\| z(t)\|_{l^2 \dot W^{1,1}(T,\infty; L^2_x)} \lesssim M T^{-\delta}  (1+M^2 \log T).
\]
Here we cannot interchange the $l^2$ and the $\dot W^{1,1}$
norm. However, we can do it if we relax $\dot W^{1,1}$ to the space
$V^2$ of functions with bounded $2$ variation,
\[
l^2 \dot W^{1,1}(T,\infty; L^2_x) \subset l^2 V^2(T,\infty; L^2_x) \subset
 V^2(T,\infty; l^2 L^2_x) =  V^2(T,\infty; L^2_x).
\]
Thus we obtain
\[
\| z(t)\|_{V^2(T,\infty; L^2_x)} \lesssim M T^{-\delta}  (1+M^2 \log T)^3.
\]

Then the desired conclusion 
\[
\| \Phi h(t)\|_{L^\infty(0,T;L^2)} + \| \Phi h(t)\|_{L^4(0,T;L^\infty)} 
 \lesssim M t^{-\delta}  (1+M^2 \log t)^3
\]
follows in view of the Strichartz embeddings 
for $V^2$ spaces,
\[
\| e^{-\frac{it \partial_x^2}{2}} z(t)\|_{L^\infty L^2} + \| \Phi h(t)\|_{L^4 L^\infty}
\lesssim \| z\|_{V^2 L^2},
\]
 see Section 4 in \cite{KochT}.

\end{proof}

\end{document}